\documentclass{article}

\usepackage[a4paper, total={6in, 8in}]{geometry}

\usepackage{amssymb}
\usepackage{amsmath}
\usepackage{amsthm}
\usepackage[english]{babel}
\usepackage{url}
\usepackage{stmaryrd}
\usepackage{bbm}
\usepackage{subcaption}
\usepackage{mathtools}

\usepackage{enumitem}
\usepackage[]{hyperref}

\usepackage[capitalise]{cleveref}

\theoremstyle{definition}
\newtheorem{definition}{Definition}[section]
\newtheorem{example}{Example}[section]

\theoremstyle{plain}
\newtheorem{theorem}{Theorem}[section]
\newtheorem{proposition}[theorem]{Proposition}
\newtheorem{lemma}[theorem]{Lemma}
\newtheorem{corollary}[theorem]{Corollary}

\newtheorem*{conjecture*}{Conjecture}

\newtheorem{maintheorem}{Theorem}

\theoremstyle{remark}
\newtheorem{remark}{Remark}[section]
\newtheorem{question}{Question}[section]

\newcommand{\N}{\ensuremath{\mathbb{N}}}
\newcommand{\Z}{\ensuremath{\mathbb{Z}}}

\newcommand{\FF}{\ensuremath{\mathcal{F}}}

\newcommand{\PP}{\ensuremath{\mathcal{P}}}

\let\myequiv\equiv
\renewcommand{\equiv}[1][]{\ensuremath{\stackrel{#1}{\myequiv}}}

\renewcommand{\phi}{\varphi}

\newcommand{\abs}[1]{\left | #1 \right |}

\newcommand{\ind}{\ensuremath{ \mathbbm{1}}}

\newcommand{\W}{\ensuremath{\mathrm{W}}}
\newcommand{\B}{\ensuremath{\mathrm{B}}}
\newcommand{\HJ}{\ensuremath{\mathrm{HJ}}}

\newcommand{\FS}[2][]{\ensuremath{\mathrm{FS}^{#1}\left(#2\right)}}

\newcommand{\seq}[2]{\ensuremath{\left ( {#1}_{#2}\right )_{#2 = 1}^\infty }}

\title{On the partition regularity of arithmetic progressions and linear equations in k-IP-sets}
\author{Raphaël Giordano \\ \href{raphael.giordano@alumni.epfl.ch}{raphael.giordano@alumni.epfl.ch}}
\date{September 2025}

\begin{document}

\maketitle

\begin{abstract}
    In this paper, we provide versions of Van der Waerden's theorem and Rado's theorem for finite colorings of IP-sets and $k$-IP-sets. Here, by an IP-set we mean a set of integers that contains all finite sums of an infinite subset of $\N$, and we define $k$-IP-sets similarly to IP-sets but allowing each summand to appear with multiplicity bounded by $k-1$. 
\end{abstract}

\tableofcontents

\section{Introduction}
    In 1927, Bartel Leendert van der Waerden demonstrated the existence of monochromatic arithmetic progressions of arbitrary length in $\N$ no matter how the latter is colored with finitely many colors.
    
   In 1933, Richard Rado provided a far-reaching generalization of Van der Waerden's theorem by classifying all systems of linear equations with integer coefficients that admit monochromatic solutions for all finite colorings of $\N$.
    
    The goal of this paper is to prove similar results in specific families of subsets of $\N$, namely IP-sets and $k$-IP-sets.\\

    \textbf{Notation:} Henceforth, let $\N := \{1,2,\ldots\}$
    and let $\llbracket a,b \rrparenthesis := [a,b)\cap \Z$.

    \begin{definition}
        Let $k \geq 2$ be a natural number. A \emph{$k$-IP-set} is a subset of $\N$ which contains a set of the form
        $$\FS[k]{\seq{x}{i}} := \left \{\sum_{i=1}^n \epsilon_i x_i : n\in \N, \epsilon_i \in \llbracket 0,k \rrparenthesis\right \}\setminus \{0\}$$
        for an increasing sequence $\seq{x}{i}$. If $k=2$, we simply call it an \emph{IP-set} and write $\mathrm{FS}$ instead of $\mathrm{FS}^2$.
    \end{definition}
    
    For convenience, all sequences are implicitly assumed to be increasing throughout the article. Sometimes, we simply write “equations" when we mean “linear equations with integer coefficients".

    \begin{definition}
        A \emph{coloring} of a set $X$ is a function $c:X \to C$ where $C$ is the set of colors. We can think of a coloring as a partition of $X$.
    
        A coloring is called \emph{finite} if $C$ is finite.
    
        A subset of $X$ is called \emph{monochromatic} if all its elements have the same color.
    \end{definition}

\subsection{Results of the paper}

Now that $k$-IP-sets are defined, our goal is to see how Van der Waerden's theorem and Rado's theorem generalize.
\begin{question}\label{questionVDW_IP}
    Can we find monochromatic arithmetic progressions in $k$-IP-sets? How does the length of the progression relate to the number $k$?
\end{question}
\begin{question}\label{questionRadoIPfixk}
    Given $k \geq 2$, which equations admit monochromatic solutions in every $k$-IP-set?
\end{question}
\begin{question}\label{questionRadoIPfixeq}
    Given an equation, for which value of $k$ does the equation admit a monochromatic solution in every $k$-IP-set?
\end{question}

We answer \cref{questionVDW_IP} with an application of the Hales-Jewett theorem to show the existence of monochromatic arithmetic progressions of length $k$ in every $k$-IP-set. This yields a version of Van der Waerden's theorem on $k$-IP-sets which turns out to be equivalent to the Hales-Jewett theorem.
\begin{maintheorem}\label{INTRO_HJ_equiv_VDWforIP}
    The following statements are equivalent:
    \begin{enumerate}
        \item the Hales-Jewett theorem.
        \item For every $k \geq 2$ and for every finite coloring of any $k$-IP-set, there exists a monochromatic arithmetic progression of length $k$.
    \end{enumerate}
\end{maintheorem}
The proof idea for this theorem is to notice the link between a combinatorial line (see \cref{parametricword}) in the context of the Hales-Jewett theorem and an arithmetic progression.

The size of the arithmetic progression matches exactly the “dimension" of the $k$-IP-set. In fact, we can construct $k$-IP-sets which do not contain arithmetic progression longer than $k$ (see \cref{boundedlengthAPinkIPset}).\\

Our next theorem provides an answer to \cref{questionRadoIPfixk} for the case $k=2$.
\begin{maintheorem}\label{INTRO_classification_regular_systems_IP}
    Let $A \in \Z^{l\times n}$ with columns $\Vec{a_1}, \ldots, \Vec{a_n} \in \Z^l$. The following statements are equivalent:
    \begin{enumerate}
        \item For every finite coloring of any IP-set, there exists a monochromatic solution to the system of equations in matrix form $A\Vec{x} = \Vec{0}$, where $\Vec{x}= (x_1,\ldots,x_n)$.
        \item There exists $m\in \N$ and $I_1, \ldots, I_m \subseteq \{1,\ldots,n\}$ such that 
        $$\bigcup_{j=1}^m I_j = \{1,\ldots,n\} \text{ and } \forall j \in \{1,\ldots,m\} \; \sum_{i \in I_j} \Vec{a_i} = \Vec{0}.$$
    \end{enumerate}

    Furthermore, we can ask the solution to be distinct in statement 1 if and only if we add the condition
    $$\forall i \neq i' \in \{1,\ldots,n\}, \exists j \in \{1,\ldots,m\} : (i \in I_j\text{ and } i' \notin I_j) \text{ or } (i \notin I_j\text{ and } i' \in I_j)$$
    in statement 2.
\end{maintheorem}
In \cref{check_criterion_columns_condition}, we show that statement 2 is stronger than the columns conditions in Rado's theorem (see \cref{columns_condition} and \cref{RadosTheorem}).

The proof of this theorem is a consequence of a corollary of Hindman's theorem. This corollary states that for every finite coloring of an IP-set, there exists a monochromatic subset which is again an IP-set.

This proof idea cannot be used for the cases $k>2$ because the corollary does not extend. There exist finite colorings of $k$-IP-sets such that we cannot find a monochromatic $k$-IP-subset. However, using an extended version of the Hales-Jewett theorem, we can always find “shifted and truncated" $k$-IP-sets which are monochromatic (see \cref{shift IPk0}). This weaker result led us to studying shift-invariant equations.

\begin{definition}
    An equation with $n$ variables on $\N$ is \textit{shift-invariant} if for every $m \in \N$ and for every $(x_1, \ldots, x_n)\in \N^n$,
    $(x_1 + m, \ldots, x_n + m)$ is a solution if and only if $(x_1, \ldots, x_n)$ is a solution.
\end{definition}

\begin{maintheorem}\label{INTRO_suff_signature_kIP}
    Consider a shift-invariant equation
     \begin{equation}\label{eqthmC}
        a_1 x_1 + \cdots + a_r x_r = a_{r+1} x_{r+1} + \cdots + a_n x_n
    \end{equation}
     with $a_i \in \N$ such that after reordering $a_1 \leq \cdots \leq a_r$ and $a_{r+1} \leq \cdots \leq a_n$.
    Consider the following statements:
    \begin{enumerate}
        \item $k > \min \left \{ \max\{a_1 + a_r, \; a_n\},\; \max\{a_r,\; a_{r+1} + a_n\} \right \}$,
        \item For every finite coloring of any $k$-IP-set, there exists a distinct solution to the equation 
        $$a_1 x_1 + \cdots + a_r x_r = a_{r+1} x_{r+1} + \cdots + a_n x_n.$$
    \end{enumerate}

    Then statement 1 implies statement 2. The converse is true for an equation with three variables, i.e. for an equation of the form 
    $$a_1 x_1 + a_2 x_2 = a_3 x_3$$
    under the assumption that $\gcd(a_1, a_2) = 1$.
\end{maintheorem}
\begin{remark}
    It follows from \cref{shift_invariance_condition} that an equation as in (\ref{eqthmC}) is shift-invariant if and only if 
    $$a_1 + \cdots + a_r = a_{r+1} + \cdots + a_n = :  \sigma.$$
    
    Noticing that $\sigma > \min \left \{ \max\{a_1 + a_r, \; a_n\},\; \max\{a_r,\; a_{r+1} + a_n\} \right \}$, we get a weaker yet more elegant version of \cref{INTRO_suff_signature_kIP} by replacing statement 1 by “$k > \sigma$".\\
\end{remark}

We could not provide a new result to answer \cref{questionRadoIPfixeq}. However, we make a couple of observations.

We first notice that modern proofs of Rado's theorem show that a monochromatic solution to a system of equations whose matrix satisfies the columns condition can be found in $(m,p,c)$-sets. These sets were introduced by Walter Deuber in \cite{Deu73}. He proved that for every $m,p,c,r \in \N$ there exist $M,P,C \in \N$ such that for every coloring of a $(M,P,C)$-set with $r$ colors, there exists a monochromatic $(m,p,c)$-subset. This result is known as Deuber's theorem.

Furthermore, we observed that for every $m,p,c \in \N$, every $k$-IP-set contains a $(m,p,c)$-set as long as $k$ is large enough.

This motivates us to consider the definition of $\overline{IP}$-sets which were first introduced by Vitaly Bergelson and Imre Z. Ruzsa in \cite{BerRuz09}. 
\begin{definition}
    An \emph{$\overline{IP}$-set} is a subset of $\N$ which contains a set of the form
    $$\left \{ \sum_{i=1}^N \epsilon_i x_i: N \in \N, \epsilon_i \in \llbracket 0,i \rrparenthesis \right\}$$
   for an increasing sequence $\seq{x}{i} \subseteq \N$.
\end{definition}

Notice that an $\overline{IP}$-set contains $k$-IP-sets for any value of $k$. The consequence of this definition and our observations is straightforward.
\begin{proposition}\label{INTRO_consolation}
    Consider a system of linear equations with integers coefficients. The following statements are equivalent:
    \begin{enumerate}
        \item For every finite coloring of $\N$, the system has a monochromatic solution.
        \item For every finite coloring of any $\overline{IP}$-set, the system has a monochromatic solution.
    \end{enumerate}
\end{proposition}

\subsection{Preliminary classical results of Ramsey theory}
Let us formally state four results that will be mentioned or used. Most of these results are proved in \cite{GRS90}.

\begin{theorem}[Van der Waerden (1927)]
    For every $r,k\in \N$ there exist $n\in \N$ such that for every coloring of $\{1,\ldots,n\}$ with $r$ colors, there exist $a,d \in \N$ such that
    $$\{a, a+d, a+2d, \ldots, a + (k-1)d\}$$
    is a monochromatic subset of $\{1,\ldots,n\}$. (See also \cite{vdW28}) The minimal value of $n$ is denoted $\W(r,k)$.
\end{theorem}

\begin{theorem}[Brauer (1928)]
    For every $k,s\in \N$ there exist $n\in \N$ such that for every coloring of $\{1,\ldots,n\}$ with $r$ colors, there exist $a,d \in \N$ such that
    $$\{a, a+d, a+2d, \ldots, a + (k-1)d\} \cup \{sd\}$$
    is a monochromatic subset of $\{1,\ldots,n\}$. The minimal value of $n$ is denoted $\B(r,k,s)$.
\end{theorem}

\begin{definition}
    A system of equations is called \emph{partition regular} on a set of numbers $X$ if for every finite coloring of $X$, there exists a monochromatic solution in $X$.
\end{definition}

\begin{definition}\label{columns_condition}
    A matrix $A$ satisfies the \emph{columns condition} if up to reordering the columns, we can split the matrix into $t$ blocks, i.e. submatrices;
    $$\begin{bmatrix}
            \quad A_1 & \bigm| & A_2 & \bigm| & \cdots & \bigm| & A_t \quad
            \end{bmatrix}$$
    such that for each block $A_i$, the sum of the vectors in $A_i$ is
    \begin{itemize}
        \item[-] the zero vector if it is the leftmost block, i.e. if $i=1$.
        \item[-] a linear combination of the columns on the left of this block, i.e. columns in $\begin{bmatrix}
            A_1 & A_2 & \cdots & A_{i-1} \end{bmatrix}$, otherwise.
    \end{itemize} 
\end{definition}

\begin{theorem}[Rado (1933)]\label{RadosTheorem}
        Consider a system of linear equations with integer coefficients in matrix form
        $$A\Vec{x} = \Vec{0}$$
        where $A \in \Z^{m \times n} $ is a matrix with non zero columns.
        Then the system is partition regular on $\N$ if and only if $A$ satisfies the columns condition.
\end{theorem}

For a single equation the theorem rephrases as follows:

\textit{
    Let $a_1, \ldots, a_n \in \Z^*$. The following are equivalent: 
    \begin{itemize}
        \item[-]  The equation $a_1 x_1 + \cdots + a_n x_n = 0$ is partition regular on $\N$.
        \item[-]  There exists $I \subseteq \{1,\ldots,n\}$ such that $\sum_{i\in I} a_i = 0$.\\
    \end{itemize}
}

\begin{definition}\label{parametricword}
    A \emph{parametric word} of the set $\llbracket 0,t \rrparenthesis^n$ is an element of the set
    $$\left (\llbracket 0,t \rrparenthesis \cup \{*\} \right)^n \setminus \llbracket 0,t \rrparenthesis^n.$$
    We call the extra symbol “$*$" a \emph{variable} of the parametric word. For example, $(1,0,*,2)$ and $(*,*,2,1)$ are parametric words of $\llbracket 0,3 \rrparenthesis^4$.
    A \emph{combinatorial line} is a set of the form 
    $$\{w(0), w(1),\ldots,w(t-1)\},$$
    where $w$ is a parametric word and $w(i)$ is obtained by replacing every occurrence of $*$ by $i$. Each parametric word corresponds to a unique combinatorial line and vice versa. For example, we have 
    $$(1,0,*,2) \leftrightarrow \{(1,0,0,2),\;(1,0,1,2),\;(1,0,2,2)\}.$$    
\end{definition}

\begin{theorem}[Hales-Jewett (1963)]
    For every $r,t \in \N$ there exists $n\in \N$ such that for every coloring of $\llbracket 0,t \rrparenthesis^n$ with $r$ colors, there exists a monochromatic combinatorial line. (See also \cite{HJ63}) The minimal value of $n$ is denoted $\HJ(r,t)$.
\end{theorem}

A \emph{combinatorial $k$-dimensional subspace} is defined analogously to a combinatorial line but the corresponding parametric word must contain $k$ distinct variables. For example, $(*_1,2,*_2,*_1)$ is the parametric word in $\llbracket 0,3 \rrparenthesis^4$ corresponding to the combinatorial plane
    $$\{(0,2,0,0),(0,2,1,0),(0,2,2,0),\phantom{\}}$$
    $$\phantom{\{}(1,2,0,1),(1,2,1,1),(1,2,2,1),\phantom{\}}$$
    $$\phantom{\{}(2,2,0,2),(2,2,1,2),(2,2,2,2)\}.$$
    
\begin{theorem}[Hales-Jewett (extended version)]\label{ExtendedHalesJewett}
    For every $r,t,k \in \N$, there exists $n \in \N$ such that for every $r$-coloring of $\llbracket 0,t \rrparenthesis^n$, there exists a monochromatic combinatorial $k$-dimensional subspace. The minimal value of $n$ is denoted $\HJ^k(r,t)$.
\end{theorem}

Using Hindman's theorem \cite{Hin74} we can show the partition regularity of IP-sets \cite{Zhou17}.
\begin{corollary}\label{IPpartregIP}
    IP-sets are partition regular in any IP-set, i.e. for every finite coloring of an IP-set $Y$, there exists $X \subseteq Y$ infinite such that $\FS{X}$ is monochromatic subset of $Y$.
\end{corollary}

\subsection{Overview of the paper}
The paper is organized as follows. In \cref{sectionHJ} we prove the equivalency between the Hales-Jewett theorem and the version of Van der Waerden's theorem in $k$-IP-sets, which is the object of \cref{INTRO_HJ_equiv_VDWforIP}. We begin \cref{sectionEQ} by classifying every system of linear equations with integer coefficients that are partition-regular on IP-sets, which is the object of \cref{INTRO_classification_regular_systems_IP}. We then verify that our criterion is coherent with the columns condition in Rado's theorem. We continue \cref{sectionEQ} by explaining the difficulties we encounter when dealing with the general cases of $k$-IP-sets. This prompts us to consider shift-invariant equations which eventually leads to \cref{INTRO_suff_signature_kIP}. 

\subsection{Acknowledgments}
This paper grew from the author's master thesis at EPFL\footnote{École polytechnique fédérale de Lausanne, Switzerland}. The author expresses his sincere gratitude to his supervisor, Professor Florian K. Richter, for his guidance and advice that carried him through the stages of his research. The author would also like to reiterate his thanks for the proofreading, which greatly contributed to the quality of this article.

\section{Equivalency of the Hales-Jewett theorem with Van der Waerden's theorem in \texorpdfstring{$k-$}-IP-sets}\label{sectionHJ}
In this section, we prove \cref{INTRO_HJ_equiv_VDWforIP} which answers \cref{questionVDW_IP}.

\begin{theorem}\label{HJ_equiv_VDWforIP}
    The following statements are equivalent:
    \begin{enumerate}
        \item the Hales-Jewett theorem.
        \item Van der Waerden's theorem in $k$-IP-sets: for every $k \geq 2$ and for every $k$-IP-set which is finitely colored, there exists a monochromatic arithmetic progression of length $k$.
    \end{enumerate}
\end{theorem}

The implication $(1) \implies (2)$ is an unpublished result that Florian Richter left to the author as a task to start his master thesis.
\begin{proof}[Proof of $(1) \implies (2)$]
    The observation to make is that the function
    $$ \llbracket 0,k \rrparenthesis^{< \omega} \setminus \{(0,\ldots,0)\}\to \FS[k]{\seq{x}{i}}, \quad (\epsilon_1,\ldots,\epsilon_n) \mapsto \sum_{i=1}^n \epsilon_i x_i$$
    maps combinatorial lines to arithmetic progressions. Indeed, if $I$ is the set of indices on which the variable of the corresponding parametric word occurs, then the common difference is given by $\sum_{i\in I} x_i$ and the initial term is given by $\sum_{i \notin I} \epsilon_i x_i$.
    The rest of the proof consists of an application of the Hales-Jewett theorem.
\end{proof}

Conversely, we cannot say that a $k$-term arithmetic progression maps back to a parametric word in general. First, it could simply happen that the inverse of the function is not defined if the latter is not bijective. and even if it was bijective, it would not be sufficient. For example in $\FS{(3^i)_{i=0}^\infty}$, the $3$-term arithmetic progression $\{5,7,9\}$ maps back to the set
$$\{(2,1,0,\ldots,0), \; (1,2,0,\ldots,0), \; (0,0,2,0,\ldots,0)\}$$
which cannot be described with a parametric word. However, choosing the generating sequence sparse enough ensures it. This is the topic of the following lemma, which shows the implication $(2) \implies (1)$.

\begin{lemma}\label{combinatorialAP}
    For every $k\geq 3$, there exists a sequence $\seq{x}{i} \subseteq \N$ such that $k$-term arithmetic progressions in $\FS[k]{\seq{x}{i}}$ map back to combinatorial lines.\footnote{Implicitly, the lemma requires that the map is bijective.}
\end{lemma}
\begin{proof}
    Let us call an arithmetic progression \emph{combinatorial} if it maps back to a combinatorial line. We construct recursively the sequence by setting $x_1 = 1$ and $x_{n+1} > 2 \max \left (\FS[k]{(x_i)_{i=1}^n}\right )$ and show by induction on $n$ that $\FS[k]{(x_i)_{i=1}^n}$ does not contain non combinatorial arithmetic progressions. This choice of $\seq{x}{i}$, ensures that the map is bijective.\\

    If $n=1$, then $\FS[k]{(x_i)_{i=1}^n}= \{1,2,\ldots, k-1\}$ which does not contain a $k$-term arithmetic progression. Eventually, if we accept to add 0 to this set, then it would form a combinatorial arithmetic progression, with corresponding parametric word $*$.\\

    Suppose that $\FS[k]{(x_i)_{i=1}^n}$ does not contain non combinatorial arithmetic progressions. For $j \in \llbracket 0,k \rrparenthesis$, let 
    $$A_j = j x_{n+1} + \FS[k]{(x_i)_{i=1}^n}.$$
    Notice that these sets are respectively contained in intervals of constant length which is smaller than the gaps between them.
   Furthermore, notice that  $$A_0 \cup \cdots \cup A_{k-1} = \FS[k]{(x_i)_{i=1}^{n+1}}$$ and that 
    $$\sum_{i=1}^{n+1} \epsilon_i x_i \in A_j \iff \epsilon_{n+1}= j.$$ 
    
    Suppose that $\FS[k]{(x_i)_{i=1}^{n+1}}$ contains a $k$-term arithmetic progression of the form
    $$a, a+d , \ldots, a + (k-1)d.$$
    
    If this progression is contained in only one $A_j$, then shifting it by $-j x_{n+1}$ yields an arithmetic progression in $\FS[k]{(x_i)_{i=1}^{n}}$ which is combinatorial by the induction hypothesis. This shifted arithmetic progression can be described by a parametric word with $0$ on its $(n+1)$-th position. Replacing this 0 by $j$ yields a combinatorial word describing our initial arithmetic progression.\\

    If not, then we deduce that $d \geq x_{n+1}$ and that $a+ jd \in A_j$ for $j=0,\ldots,k-1$ due to the size of the gaps. It follows that 
    $$\{a+ j(d-x_{n+1}): 0 \leq j <k\}$$
    is either a $k$-term arithmetic progression in $\FS[k]{(x_i)_{i=1}^{n}}$ or a constant sequence of length $k$.
    
    In the first subcase, this arithmetic progression can be again described by a parametric word with 0 in its $(n+1)$-th position by the induction hypothesis. Replace this 0 by the variable $*$ of the combinatorial word so that the new word describes our initial arithmetic progression. 
    
    In the second subcase, it means that the sequence can be described by a constant word, i.e. without variable, on the first $n$ positions and that $d= \epsilon_{n+1}$. Replace the $(n+1)$-th position by $*$ once again so that the new word describes the initial arithmetic progression. \\

    In any case, the arithmetic progressions in $\FS[k]{(x_i)_{i=1}^{n+1}}$ are combinatorial.
    
\end{proof}

Furthermore, notice that a $k$-IP-set as in this proof cannot contain an arithmetic progression of length greater than $k$. By contradiction, if it did, then in particular it would contain an arithmetic progression of length $k+1$. Its first $k$ terms and its last $k$ terms would form two arithmetic progressions of length $k$ which correspond respectively to two distinct parametric words $w(x), \Tilde{w}(x)$ which satisfy $w(1)= \Tilde{w}(0)$ and $w(2)= \Tilde{w}(1)$. But this implies that the set of indices in which the variable $x$ occurs in these two words must be the same. This contradicts the fact that the map $(\epsilon_1,\ldots,\epsilon_n) \mapsto \sum_{i=1}^n \epsilon_i x_i$ is bijective. The next corollary therefore follows.

\begin{corollary}\label{boundedlengthAPinkIPset}
    For every $k\geq 2$, there exists a $k$-IP-set that does not contain arithmetic progressions of length greater than $k$.
\end{corollary}

\section{Partition regular linear equations with integer coefficients in \texorpdfstring{$k$}{k}-IP-sets}\label{sectionEQ}
In this section, we present our results related to \cref{questionRadoIPfixk}. We first treat the case of IP-set, i.e. when $k=2$, before treating the case of $k$-IP-sets in general.

\subsection{Classification of partition regular linear equations with integer coefficients in IP-sets}\label{subsectionREGEQIP}
In this subsection, we prove \cref{INTRO_classification_regular_systems_IP}.\\

It is tautological that if every IP-set contains a monochromatic subset of a certain type, such as an arithmetic progression of fixed length or a solution to some equation, then every IP-set without any coloring simply contains a subset of this type. What is less evident and perhaps surprising is that the converse also holds.
\begin{lemma}\label{neat_fact}
    Let $\FF \subseteq \PP(\N)$. The following statements are equivalent:

    \begin{enumerate}
        \item $\FF$ is partition regular on every $IP$-set, i.e. for every finite coloring of an IP-set there exists a monochromatic subset which belongs to $\FF$.
        \item Every $IP$-set contains a subset $A \in \FF$.
    \end{enumerate}
\end{lemma}
\begin{proof}
    $(1) \implies (2)$ is obvious, while $(2) \implies (1)$ follows from \cref{IPpartregIP}.
\end{proof}

With this neat fact in hand, one can already prove the partition regularity of a couple of equations, such as Schur's equation:
$$x+y = z.$$
Indeed, $y_1,y_2, y_1+y_2 \in \FS{\{y_1,y_2\}}$ yields a solution. Let us prove \cref{INTRO_classification_regular_systems_IP} now.

\begin{theorem}[Rado's theorem on IP-sets]\label{classification_regular_systems_IP}
    Let $A \in \Z^{l\times n}$ with columns $\Vec{a_1}, \ldots, \Vec{a_n} \in \Z^l$. The following statements are equivalent:
    \begin{enumerate}
        \item The system of equations in matrix form $A\Vec{x} = \Vec{0}$, where $\Vec{x}= (x_1,\ldots,x_n)$, is partition regular on every IP-set.
        \item There exist $m\in \N$ and $I_1, \ldots, I_m \subseteq \{1,\ldots,n\}$ such that 
        $$\bigcup_{j=1}^m I_j = \{1,\ldots,n\} \text{ and } \forall j \in \{1,\ldots,m\} \; \sum_{i \in I_j} \Vec{a_i} = \Vec{0}.$$
    \end{enumerate}
\end{theorem}
\begin{remark}
    The version of this theorem for a single equation is essentially obtained by replacing the word “column" with the word “coefficient".
\end{remark}
        
        Similarly as before, showing the implication $(1) \implies (2)$ requires to consider an IP-set generated by a sparse enough sequence. This motivates the following lemma before moving to the proof of the theorem.
        \begin{lemma}\label{uniqueness B negative}
            If $\sum_{i=1}^{m} \gamma_i B^i = \sum_{i=1}^{m} \delta_i B^i$ for some $B\in \N$ and $\gamma_i, \delta_i \in \Z$ such that $\abs{\gamma_i},\abs{\delta_i} < \frac{B}{2}$, then for every $i \in \{1,\ldots,m\}$, $\gamma_i = \delta_i$.
        \end{lemma}
        The proof relies on the uniqueness of a number's representation in base $B$.
        \begin{proof}
            First, we reorder the terms in the equality such that all summands are non-negative. Formally, this yields

            $$\sum_{i=1}^{m} \underbrace{\left ( \max(\gamma_i,0) + \abs{\min(\delta_i,0) } \right)}_{< 2 \frac{B}{2}= B }  B^i 
            =\sum_{i=1}^{m} \underbrace{ \left ( \max(\delta_i,0) + \abs{\min(\gamma_i,0) } \right)}_{< 2 \frac{B}{2}= B} B^i . $$

            Thus for $i \in \{1,\ldots,m\} $, we have
            $\max(\gamma_i,0) + \abs{\min(\delta_i,0)}  = \max(\delta_i,0) + \abs{\min(\gamma_i,0)}$. In any case of the signs of $\gamma_i$ and $\delta_i$, we deduce that $\gamma_i = \delta_i$.
        \end{proof}
    \begin{remark}\label{sparse_enough}
        In \cref{uniqueness B negative}, we can replace the powers $B^i$ by a sequence of numbers $(y_i)_{i=1}^m$ which increases sufficiently fast.
    \end{remark}
        
\begin{proof}[Proof of \cref{classification_regular_systems_IP}]
    $(2) \implies (1):$ Let $I_1,\ldots, I_m \subseteq \{1,\ldots,n\}$ such that 
        $$\bigcup_{j=1}^m I_j = \{1,\ldots,n\} \text{ and } \forall j \in \{1,\ldots,m\} \; \sum_{i \in I_j} \Vec{a_i} = \Vec{0}.$$

        In light of \cref{neat_fact}, we show that $\FS{(y_i)_{i=1}^m}$ contains a solution for any $y_1,\ldots,y_m \in \N$. We set 
        $$x_i = \sum_{j=1}^m \ind_{I_j}(i) y_j.$$
        Notice that $x_i \neq 0$ because $i\in I_j$ for some $j$ by assumption. Then we have
    \begin{align*}
        A\Vec{x} &= \sum_{i=1}^n x_i \Vec{a_i}
            = \sum_{i=1}^n \left (\sum_{j=1}^m \ind_{I_j}(i) y_j \right) \Vec{a_i}
            = \sum_{j=1}^m y_j \left(\sum_{i=1}^n \ind_{I_j}(i) \Vec{a_i}\right)\\
            &= \sum_{j=1}^m y_j\underbrace{\left (\sum_{i\in I_j} \Vec{a_i} \right )}_{=\Vec{0}}
            =\Vec{0}
    \end{align*}

    $(1) \implies (2):$ By partition regularity, there exist $m \in \N, x_1, \ldots, x_n \in \FS{(B^j)_{j=1}^m}$ where
    $$B= 2n  \max\left\{\abs{a_{i,j}}: 1 \leq i\leq n, \; 1 \leq j \leq m \right\} + 1$$
    such that
    $$A\Vec{x}= x_1 \Vec{a_1} + \cdots + x_n \Vec{a_n} =\Vec{0}.$$
    Write $$x_i= \sum_{j=1}^m \epsilon_{i,j} B^j$$ with $\epsilon_{i,j} \in \{0,1\}$ for $i=1,\ldots,n,\; j= 1,\ldots,m$. Then
    $$\sum_{i=1}^n \left( \sum_{j=1}^m \epsilon_{i,j} B^j\right) \Vec{a_i} = \Vec{0}$$
    and thus
    $$\sum_{j=1}^m B^j\left (\sum_{i=1}^n \epsilon_{i,j} \Vec{a_i} \right )= \Vec{0}.$$
    
    In particular for each line $ t \in \{1,\ldots,l\}$, we get
    $$\sum_{j=1}^m B^j\left (\sum_{i=1}^n \epsilon_{i,j} a_{i,t} \right )= 0.$$
    Since $$\abs{\sum_{i=1}^n   \epsilon_{i,j} a_{i,t}} \leq n \max\{\abs{a_{i,j}}: 1 \leq i\leq n, \; 1 \leq j \leq m\} < \frac{B}{2},$$
    it follows by \cref{uniqueness B negative} that for every $j\in \{1,\ldots,m\}$,
    $$\sum_{i=1}^n \epsilon_{i,j} a_{i,t} = 0.$$
    The equalities on each line imply
    $$\sum_{i=1}^n \epsilon_{i,j} \Vec{a_i}= \Vec{0}.$$
    
    For $j=1,\ldots,m$, let $$I_j = \{i \in \{1,\ldots,n\}: \; \epsilon_{i,j}= 1\}.$$ By construction,
    $$\sum_{i\in I_j} \Vec{a_i} = \sum_{i=1}^n  \epsilon_{i,j} \Vec{a_i} = \Vec{0}.$$
    Let $i \in \{1,\ldots,n\}$. Since $x_i \neq 0$, there exists $j\in \{1,\ldots,m\}$ such that $\epsilon_{i.j}= 1$ which implies that $i \in I_j$. Therefore $$\bigcup_{j=1}^m I_j = \{1,\ldots,n\}$$ and the criterion is satisfied.    
\end{proof}
\begin{corollary}\label{extra_distinct_solution}
    In \cref{classification_regular_systems_IP}, we can ask the solution to be distinct in statement 1 if and only if we add the condition
    $$\forall i \neq i' \in \{1,\ldots,n\}, \exists j \in \{1,\ldots,m\} : (i \in I_j\text{ and } i' \notin I_j) \text{ or } (i \notin I_j\text{ and } i' \in I_j)$$
    in statement 2.
\end{corollary}
\begin{proof}
    In light of \cref{sparse_enough}, if $(y_i)_{i=1}^m$ increases sufficiently fast, then for every $i,i'\in \{1,\ldots,n\}$ we have the equivalence
    $$\left ( \sum_{j=1}^m \ind_{I_j}(i) y_j = \sum_{j=1}^m \ind_{I_j}(i') y_j \right ) \iff \left (  \forall j\in \{1,\ldots,m\}, i \in I_j \iff i' \in I_j\right ).$$
\end{proof}
\cref{classification_regular_systems_IP} and \cref{extra_distinct_solution} show \cref{INTRO_classification_regular_systems_IP}.

\subsubsection{Check that the criterion implies the satisfaction of the columns condition}\label{check_criterion_columns_condition}
Notice that if a system of equations is partition regular on every IP-set, then it is in particular partition regular on $\N$. Thus if a matrix satisfies the criterion in the statement of \cref{classification_regular_systems_IP}, then by Rado's theorem it must satisfy the columns condition. Let us verify this property as a sanity check.

\begin{proposition}
    Let $A \in \Z^{l\times n}$ with non-zero columns $\Vec{a_1}, \ldots, \Vec{a_n} \in \Z^l$. If there exist $m\in \N$ and $I_1, \ldots, I_m \subseteq \{1,\ldots,n\}$ such that 
        $$\bigcup_{j=1}^m I_j = \{1,\ldots,n\} \text{ and } \forall j \in \{1,\ldots,m\} \; \sum_{i \in I_j} \Vec{a_i} = \Vec{0},$$
    then $A$ satisfies the columns condition.
\end{proposition}
\begin{proof}
    Without loss of generality, we can assume that 
    $$I_k \nsubseteq \bigcup_{j=1}^{k-1} I_j$$
    because otherwise we could remove $I_k$ safely from the family $I_1,\ldots, I_m$ and still satisfy the assumption.\\

    Define recursively $\Tilde{I_1}, \ldots, \Tilde{I_m} \subseteq \{1,\ldots,n\}$ by $\Tilde{I_1} = I_1$ and
    $$\Tilde{I_k} = I_k \setminus \bigcup_{j=1}^{k-1} \Tilde{I_j}.$$
    Then $\Tilde{I_1}, \ldots, \Tilde{I_m}$ yield a partition of $\{1,\ldots,n\}$ and verify
    $$\sum_{i\in \Tilde{I_1}} \Vec{a_i} = \Vec{0}$$
    and
    $$\sum_{i\in \Tilde{I_k}} \Vec{a_i} = \underbrace{\sum_{i\in {I_k}} \Vec{a_i}}_{=\Vec{0}} - \sum_{j=1}^{k-1}\sum_{i\in \Tilde{I_j} \cap I_k} \Vec{a_i} \in \mathrm{span}\left\{ \Vec{a_i} : i \in \bigcup_{j=1}^{k-1} \Tilde{I_j}\right\}.$$
    This shows that the matrix $A$ satisfies the columns condition: the splitting into blocks is given by the sets $\Tilde{I_1}, \ldots, \Tilde{I_m}$ after reordering the columns in such a way that these sets are consecutive intervals.
\end{proof}

\subsection{Partition regular equations in \texorpdfstring{$k$}{k}-IP-sets}\label{subsectionKIP}
In this subsection, we prove \cref{INTRO_suff_signature_kIP}.\\

The main obstacle when dealing with $k$-IP-sets for $k>2$ is that \cref{neat_fact}, which follows from \cref{IPpartregIP}, does not generalize.

\begin{lemma}
    For $k\geq 3$, the family of $k$-IP-sets is not partition regular.
\end{lemma}
    \begin{proof}
        By contradiction, if $k$-IP-sets were partition regular, it would in particular imply that for every finite coloring of $\N$ there is a monochromatic set of the form $\{a, 2a\}$ which provides a solution to the equation $2x_1 = x_2$. But this would imply that this equation is partition regular on $\N$ contradicting Rado's theorem.  
    \end{proof}
However, there exists a weaker statement which is correct and still interesting on its own.

    \begin{proposition}\label{shift IPk0}
        Let $k\geq 2$ and $A\subseteq \N$ be a $k$-IP-set. For every finite coloring of $A$ and every $t\in \N$, there exist $m \in \N \cup \{0\}$ and $y_1,\ldots, y_t \in \N$ such that $m + \FS[k]{(y_i)_{i=1}^t}$ is monochromatic.
    \end{proposition}
    To prove this statement, we use the extended version of the Hales-Jewett theorem. 
    \begin{proof}[Proof of \cref{shift IPk0}]
        Let $\FS[k]{\seq{x}{i}} \subseteq A$. Fix a finite coloring $\chi: \FS[k]{\seq{x}{i}} \to \llbracket 0,r \rrparenthesis$ and $t \in \N$.\\
        
        Let $n = \HJ^t(r,k)$ as in the extended version of the Hales-Jewett theorem. This induces a coloring $\chi'$ of $\llbracket 0,k \rrparenthesis^n$ given by $\chi' = \chi \circ \phi$, where $\phi(\epsilon_1, \ldots, \epsilon_n) = \sum_{i=1}^n \epsilon_i x_i$. By the extended version of the Hales-Jewett theorem, there exists a monochromatic $t$-subspace of $\llbracket 0,k \rrparenthesis^n$. By construction, this implies the existence of $I_1,\ldots,I_t \subseteq \{1,\ldots,n\}$ disjoint and non-empty, as well as $\alpha_i \in \llbracket 0,k \rrparenthesis$ for $i\notin \bigcup_{i=1}^t I_i$ such that
        $$\left \{ \epsilon_1 \sum_{i \in I_1} x_i + \cdots + \epsilon_t \sum_{i \in I_t} x_i + \sum_{i \notin \bigcup_{i=1}^t I_i} \alpha_i x_i \in \FS[k]{\seq{x}{i}}: \epsilon_1,\ldots,\epsilon_t \in \llbracket 0,k \rrparenthesis \right \}$$
        is monochromatic.\\

        Set $y_j= \sum_{i \in I_j} x_i$ for $j=1,\ldots,t$ and $m = \sum_{i \notin \bigcup_{i=1}^t I_i} \alpha_i x_i$. Then $m + \FS[k]{(y_i)_{i=1}^t}$ is monochromatic as desired.
    \end{proof}
    \begin{remark}\label{remark_good_sum}
        By refining the sequence $\seq{x}{i}$ if needed, we can ensure that the set $\FS[k]{(y_i)_{i=1}^t}$ has distinct sums, meaning that he map $(\epsilon_1,\ldots,\epsilon_t) \mapsto \sum_{i=1}^t \epsilon_i y_i$ is injective.
    \end{remark}

    The good news is that we almost have monochromatic truncated $k$-IP-sets. The bad news is that we cannot get rid of the number $m$ which shifts the set. Therefore, it is good to ask when this shift is not significant. This motivates us to define the notion of \emph{shift invariance}.
 
\begin{definition}
    We say that  a family $\FF \subseteq \PP(\N)$ is \textit{shift-invariant} if for every set $A \subseteq \N$ and $m \in \N$,
    
    $$A \in \FF \iff A + m \in \FF.$$

    An equation on $\N$ is \textit{shift-invariant} if its set of solutions is shift-invariant.
\end{definition}
\begin{lemma}\label{shift_invariance_condition}
    An equation of the form $$a_1 x_1 + \cdots + a_n x_n = 0, \quad a_i \in \Z^*$$
    is shift-invariant if and only if $a_1 + \cdots + a_n = 0$.
\end{lemma}
\begin{proof}
    The criterion follows from the fact that
    $$a_1 (x_1 + m) + \cdots + a_n (x_n + m) = 0 \iff (a_1 x_1 + \cdots + a_n x_n)  + (a_1 + \cdots + a_n)m= 0$$
    for every $m \in \N$.
\end{proof}

Notice that a shift-invariant equation always admits trivial solutions, i.e. setting $x_1= \cdots = x_n = y$ for any $y \in \N$. Therefore, we cannot answer \cref{questionRadoIPfixk} with this method. This is why we focus on the set of distinct solutions and ask what conditions ensure that a shift-invariant equation admits a monochromatic distinct solution for every finite coloring of a $k$-IP-set. This motivates the following result which is analogous to \cref{neat_fact}.

\begin{lemma}\label{pseudo_neat_fact}
    Let $a_1 x_1 + \cdots + a_n x_n = 0$ be a shift-invariant equation. The following statements are equivalent:
    \begin{enumerate}
        \item For every finite coloring of a $k$-IP-set, there exists a monochromatic distinct solution.
        \item There exists $t \in \N$ such that for every $y_1, \ldots ,y_t\in \N$, if the set $\FS[k]{(y_j)_{j=1}^t}$ has distinct sums then it contains a distinct solution.
    \end{enumerate}

\end{lemma}

\begin{proof}
    $(2) \implies (1)$: Fix a coloring of a $k$-IP-set. By \cref{shift IPk0}, there exists a monochromatic set of the form
    $$m + \FS[k]{(y_j)_{j=1}^t}.$$
    Moreover, we can ensure $\FS[k]{(y_j)_{j=1}^t}$ has distinct sums by refining the generating sequence of the $k$-IP-set (see \cref{remark_good_sum}). By assumption (2), the set
    $$\FS[k]{(y_j)_{j=1}^t}$$
    contains a distinct solution. By shift invariance, so does the set
    $$m + \FS[k]{(y_j)_{j=1}^t}$$
    which shows (1).\\

    $(1) \implies (2)$: Let $B= kn \max_{1 \leq i \leq n} \{\abs{a_i}\} + 1$. By assumption (1), there is a monochromatic distinct solution in $\FS[k]{(B^j)_{j=1}^t}$ for some $t \in \N$.
    Write $$x_i= \sum_{j=1}^t \epsilon_{i,j} B^j$$ with $\epsilon_{i,j} \in \llbracket 0,k \rrparenthesis$ for $i=1,\ldots,n,\; j= 1,\ldots,t$. Then
    $$\sum_{i=1}^n a_i \sum_{j=1}^t \epsilon_{i,j} B^j = 0$$
    and thus
    $$\sum_{j=1}^t \left (\sum_{i=1}^n a_i  \epsilon_{i,j} \right ) B^j = 0.$$
    Since $$\abs{\sum_{i=1}^n a_i  \epsilon_{i,j}} < k n \max\{\abs{a_1},\ldots,\abs{a_n}\} = \frac{B}{2},$$
    it follows by \cref{uniqueness B negative} that for every $j\in \{1,\ldots,t\}$,
    $$\sum_{i=1}^n a_i  \epsilon_{i,j} = 0.$$
    Therefore, for every $y_1, \ldots, y_j \in \N$,
    $$\sum_{j=1}^t \left (\sum_{i=1}^n a_i  \epsilon_{i,j} \right ) y_j = 0$$
    and by switching the sum again
    $$\sum_{i=1}^n a_i \sum_{j=1}^t \epsilon_{i,j} y_j = 0$$
    which yields a solution in $\FS[k]{(y_j)_{j=1}^t}$. If furthermore the set $\FS[k]{(y_j)_{j=1}^t}$ has distinct sums then the solution is distinct.
\end{proof}
\begin{remark}
    We also have a weaker statement with a valid proof if we replace the word “distinct" by “non-trivial".
\end{remark}

With this result in hand, our goal is to look for distinct/non-trivial solutions of shift-invariant equations in sets of the form $\FS[k]{(y_j)_{j=1}^t}$ for some $t \in \N$. Notice that a shift-invariant equation with coefficients in $\Z^*$ has at least two coefficients: one that is positive and another one that is negative. If the equation has only two variables and thus two coefficients, then the absolute values of those are the same and there is no non-trivial solution. Therefore the first interesting case is an equation with three variables of the form
$$ax + by= (a+b)z.$$

\begin{proposition}\label{shift_invariant_basecase}
   Let $a,b \in \Z^*$ such that $\gcd(a,b)= 1$. Let $k \geq 2$. The following are equivalent:
    \begin{enumerate}
        \item Every $k$-IP-set contains a distinct\footnote{For a shift-invariant equation with three variables “non-trivial" and “distinct" mean the same.} solution to the equation $$ax+ by= (a+b)z.$$
        \item $k > a+b$.
    \end{enumerate}
\end{proposition}
\begin{proof}
    Since $\gcd(a,b)= \gcd(a,a+b)= \gcd(b,a+b)$, we can assume without loss of generality that $a,b\in \N$ by reorganising the equation.\\

    (2) $\implies$ (1): Let $A \supseteq \FS[k]{\seq{y}{i}}$ be a $k$-IP-set. Set
    $$x= y_1, \quad y=y_1 + (a+b)y_2,\quad z = y_1 + b y_2.$$
    It is easy to check that $x,y,z$ form a distinct solution. (We do not need the assumption $\gcd(a,b)=1$ to prove this direction.) \\
    
    (1) $\implies$ (2): In particular, there exists $n\in \N$ minimal such that the set $\FS[k]{(B^i)_{i=1}^n}$, where $B=k(a+b)$, contains a distinct solution. Write
    $$x = \sum_{i=1}^{n} \alpha_i B^i, \quad y = \sum_{i=1}^{n} \beta_i B^i, \quad z = \sum_{i=1}^{n} \epsilon_i B^i.$$
    By shift invariance of the equation, we can assume without loss of generality that $x,y,z \in \FS[k]{(B^i)_{i=1}^n} \cup \{0\}$ and that $\alpha_{n} \cdot \beta_{n} \cdot \epsilon_{n} = 0$, i.e. one of them is 0. Notice however that not all of them can be 0. Indeed, if $n=1$, this would imply that the solution is trivial and if $n>1$, then it would contradict the minimality of $n$.
    
    We have
    $$a\sum_{i=1}^{n} \alpha_i B^i +b  \sum_{i=1}^{n} \beta_i B^i= (a+b) \sum_{i=1}^{n} \epsilon_i B^i$$
    which implies
    $$\sum_{i=1}^{n} (a \alpha_i + b \beta_i )B^i= \sum_{i=1}^{n}   (a+b) \epsilon_i B^i.$$
    Since $0 \leq (a \alpha_i + b \beta_i )  < k(a+b)$ and $0 \leq (a+b) \epsilon_i  < k(a+b)$, we deduce by uniqueness of a number's representation in base $B$ that 
    $$a \alpha_n + b \beta_n = (a+b)\epsilon_n.$$
    
    We analyse all the possible cases:
    \begin{itemize}      
        \item[-]  If $\alpha_n=0$, then $b\beta_n = (a+b)\epsilon_n$ and thus $(a+b)$ divides $\beta_n$ because $\gcd(a+b,b)= \gcd(a,b)= 1$. Therefore, $k>\beta_n \geq a+b$.
        
        \item[-]  If $\beta_n=0$, then we similarly deduce that $k> a+ b$.
        
        \item[-]  If $\epsilon_n=0$, then $a\alpha_n+b\beta_n=0$ and thus $\alpha_n=\beta_n=0$, which contradicts the minimality of $n$, so the case $\epsilon_n=0$ is not possible.
    \end{itemize}

   Therefore we conclude that $k> a+b$.
\end{proof}
\begin{remark}\label{remark_shift_inv_basecase}
    In this proof, we actually showed the existence of a distinct solution in every set of the form $\FS[k]{\{y_1,y_2\}}$.\\ 
\end{remark}

The idea now would be to generalize this criterion. Consider a shift-invariant equation of the form
$$a_1 x_1 + \cdots + a_r x_r = a_{r+1} x_{r+1} + \cdots + a_n x_n$$
with $a_i \in \N$ for $i\in \{1,\ldots,n\}$. (Thus $a_1 + \cdots + a_r = a_{r+1}+ \cdots + a_n$.)

It would be nice to determine the \emph{signature} of an equation, i.e. a number $\sigma$ depending on its coefficients such that the equation has a distinct/non-trivial solution in any $k$-IP-set if and only if $k> \sigma$. In the previous proposition, the signature would be given by $$\sigma := a + b.$$
Notice however that we needed an assumption on the coefficients: the latter needed to be coprime.

A first elegant guess is to claim that the signature is given by
$$\sigma:= a_1 + \cdots + a_r = a_{r+1} + \cdots + a_n .$$
A second elegant guess is to claim that the signature is given by
$$\sigma:= \max_{1\leq i \leq n}\{a_i\}.$$
Both these guesses are wrong if we consider the case of non-trivial solution. Indeed, the equation 
$$x_1 + 2x_2 + 3 x_3 = 6 x_4$$
has a solution in every $3$-IP-sets given by 
$$x_1 = x_2 = 2 y_1, \quad x_3 = 2 y_2, \quad x_4 = y_1 + y_2.$$

This is not a really big surprise because what happens is that the initial equation reduces to the equation
$$3 x_1 + 3 x_3 = 6 x_4$$
which is equivalent to
$$x_1 + x_3 = 2x_4.$$
Distinct solutions of this equation are precisely given by 3-term arithmetic progressions. For this reason, it is more relevant to focus on the case of distinct solutions only. Sadly, the guesses are also wrong for distinct solutions as the next example shows.
\begin{example}
    The shift-invariant equation
    $$3x_1+ 5x_2 + 11 x_3 = 19 x_4$$
    has a solution in every set of the form $\FS[8]{\{y_1,y_2\}}$ given by 
    $$x_1 = y_2 + y_1, \quad x_2 = y_2 + 7 y_1, \quad x_3 = y_2, \quad x_4 = y_2 + 2 y_1.$$
    Notice that the coefficients are pairwise coprime, which is even stronger than coprimality.
\end{example}

In the example, the solution uses the equality
$$3\cdot 1 + 5 \cdot 7 + 11\cdot 0 = 19 \cdot 2.$$
For an equation of the form
$$a_1 x_1 + \cdots + a_{n-1} x_{n-1} = a_n x_n$$
with $a_i \in \N$, we can similarly deduce a distinct solution in every set of the form $\FS[k]{\{y_1,y_2\}}$ if this equation has a distinct solution in $\llbracket 0,k \rrparenthesis$. Indeed, multiply this solution by $y_1$ and shift it by $y_2$.

For an equation with more than one term on the right hand side of the equation, things start to get chaotic. In the same way that we built a solution for IP-sets, the idea would be to consider a covering
$$\bigcup_{j=1}^m I_j = \{1,\ldots,n\}$$
such that for $j=1,\ldots,m$, the equations
$$\sum_{i\in I_j} c_i x_i$$
have distinct solutions in $\llbracket 0,k \rrparenthesis$ (notice that 0 is allowed). To build a solution in $y_{m+1} + \FS[k]{(y_j)_{j=1}^m}$, multiply each of these solutions by $y_j$ respectively, add them together and shift the whole thing by $y_{m+1}$. This solution is distinct if $\FS[k]{(y_j)_{j=1}^m}$ has distinct sums. The challenge consists of finding the smallest $k$ for which all of this is possible, which is a laborious combinatorial task.

We interrupt the study of distinct solutions for shift-invariant equations by giving a sufficient criterion.

\begin{proposition}\label{sufficient_bound_shift_inv}
     Consider a shift-invariant equation of the form
     $$a_1 x_1 + \cdots + a_r x_r = a_{r+1} x_{r+1} + \cdots + a_n x_n$$
     with $a_i \in \N$ such that after reordering
     $a_1 \leq \cdots \leq a_r $ and $ a_{r+1} \leq \cdots \leq a_n.$

    If $k > \min \left \{ \max\{a_1 + a_r, \; a_n\},\; \max\{a_r,\; a_{r+1} + a_n\} \right \},$ then for every $y_2, \ldots ,y_n\in \N$, if the set $\FS[k]{(y_j)_{j=2}^n}$ has distinct sums then it contains a distinct solution to the equation.

\end{proposition}
\begin{remark}
    Notice that this result is consistent with \cref{shift_invariant_basecase}.
\end{remark}
\begin{proof}

    For $j=2,\ldots,r$, we define
    $$\Vec{\epsilon_j} = (0,a_j,\ldots, a_j, a_1 + a_j, a_j, \ldots, a_j) \in \Z^n,$$
    where $a_1+a_j$ occurs on the $j$-th position, and for $j=r+1,\ldots,n$, we define
    $$\Vec{\epsilon_j} = (a_j,0\ldots, 0, a_1 , 0, \ldots, 0) \in \Z^n,$$
    where $a_1$ occurs on the $j$-th position as well. We notice that these vectors solve the equation. By linearity, for every $y_2,\ldots,y_n \in \N$, the vector
    $$\Vec{x} := \sum_{j=2}^n y_j \Vec{\epsilon_j}$$
    also solves the equation. Moreover, this vector has only positive entries, i.e. $\Vec{x} \in \N^n$. Denote
    $$\Vec{\epsilon_j}= (\epsilon_{j,1}, \ldots, \epsilon_{j,n})$$
    so that for $i=1,\ldots,n$,
    $$x_i = \sum_{j=2}^n  \epsilon_{j,i} y_j$$
    which by construction yields a solution in $\FS[k]{(y_j)_{j=2}^n}$ for
    $$k > \max \left \{ \epsilon_{j,i}: 2 \leq j\leq n,\; 1 \leq i\leq n\right \}= \max\{a_1 + a_r, \; a_n\}.$$
    Moreover, if $\FS[k]{(y_j)_{j=2}^n}$ has distinct sums, then the solution is distinct.\\

    By symmetry (by swapping the role of $x_1$ and $x_{r+1}$), we also have a solution in sets of the form $\FS[k]{(y_j)_{j=2}^n }$, if 
    $$k > \max\{a_r,\; a_{r+1} + a_n\}.$$
    Therefore, we can choose the inequality with less restraint on $k$, i.e. it suffices that
    $$k > \min \left \{ \max\{a_1 + a_r, \; a_n\},\; \max\{a_r,\; a_{r+1} + a_n\} \right \}.$$
\end{proof}

\cref{INTRO_suff_signature_kIP} follows from \cref{pseudo_neat_fact}, \cref{shift_invariant_basecase} and \cref{sufficient_bound_shift_inv}. 

\bibliographystyle{plain}
\bibliography{biblio}

\end{document}